\documentclass[12pt]{article}
\usepackage[a4paper, left=2cm, right=2cm, bottom=2cm, top=2cm]{geometry}
\usepackage{amsmath}
\usepackage{srcltx}
\usepackage{amsfonts}
\usepackage{amssymb}
\usepackage{lipsum}
\usepackage{psfrag}
\usepackage{verbatim}
\usepackage{graphicx} 
\usepackage{color}
\usepackage[pagewise]{lineno}
\providecommand{\U}[1]{\protect\rule{.1in}{.1in}}

\newtheorem{theorem}{Theorem}%[section]

\newtheorem{corollary}[theorem]{Corollary}

\newtheorem{definition}[theorem]{Definition}
\newtheorem{example}[theorem]{Example}

\newtheorem{lemma}[theorem]{Lemma}

\newtheorem{proposition}[theorem]{Proposition}
\newtheorem{remark}[theorem]{Remark}

\newenvironment{proof}[1][Proof]{\noindent\textbf{#1.} }{\hfill $\square$ \newline}

\newcommand{\R}{\ensuremath{\mathbb{R}}}

\newcommand{\N}{\ensuremath{\mathbb{N}}}
\newcommand{\Q}{\ensuremath{\mathbb{Q}}}
\newcommand{\Z}{\ensuremath{\mathbb{Z}}}

\begin{document}
	\title{Controllability of discrete-time linear systems on solvable Lie groups}
	\date{}
	\author{Victor Ayala$^{1}$  \hspace{3,5cm} Thiago M. Cavalheiro$^{2}$ \\
		Universidad de Tarapacá,   Chile \hspace{1cm}
		State University of Maringa,  Brazil \vspace{0,3cm}\\
		João A. N. Cossich$^{3}$    \hspace{2cm}  
		Alexandre J. Santana$^{4}$ \thanks{Parcially supported by CNPq grant n$^{o}$ 309409/2023-3}\\
		State University of Maringa,  Brazil  \hspace{1cm}
		State University of Maringa,  Brazil
		\vspace{0,5cm}\\
		$^{4}$Corresponding author. E-mail: ajsantana$@$uem.br;\\
		$^{1,2,3}$Contributing authors: vayala$@$academicos.uta.cl ,  \\ joaocossich\textsubscript{-}$@$hotmail.com , thiago\textsubscript{-}mcavalheiro$@$hotmail.com}	
	
	\maketitle
	
	\begin{abstract} 
		This study investigates the controllability property of discrete-time linear control systems on a connected solvable Lie group $G$. This notion is determined by an automorphism $f_{0}$ of  $G$, a recurrent formula given by the controls, and certain left-translation invariance. Through its generalized eigenspaces, $d(f_0)_e$  decomposes the Lie algebra $\mathfrak g$ as $\mathfrak{g} = \mathfrak{g}^+ \oplus \mathfrak{g}^0 \oplus \mathfrak{g}^-$, inducing also a $G$-decomposition $G=G^+ G^0 G^-$ at the level group. We use this framework and a result by Jakubczyk and Sontag to establish relations between the reachable set $\mathcal{R}$ and the controllable set $\mathcal{C}$, both originating from the identity $e \in G$, with the $G^0$ component. Then, we prove that for a solvable Lie group, sufficient conditions for controllability are: $\mathcal R$ and $\mathcal C$ are open sets, and every eigenvalue of $d(f_{0})$ has a norm $1$. Here, $f_{0}$ is the associated automorphism of the system.
		Furthermore, a necessary and sufficient condition for controllability was established for simple connected nilpotent Lie groups. Moreover, the whole-class discrete-time linear systems on a two-dimensional affine Lie group were built, and a necessary and sufficient condition for controllability in this group was also stated. We finish with an example of a discrete-linear control system on the connected and simply connected Heisenberg Lie group of dimension three.
	\end{abstract}
	
	

	%%%%%%%%%%%%%%%%%%%%%%%%%%%%%%%%%%%%%%%%%%%%%%%%%%%%%%%%%%%%%%%%%%%%%%%%%%%%%%%%%%%%%%%

	%%%%%%%%%%%%%%%%%%%%%%%%%%%%%%%%%%%%%%%%%%%%%%%%%%%%%%%%%%%%%%%%%%%%%%%%%%%%%%%%%%%%%%%
	
	\section{Introduction}

	%%%%%%%%%%%%%%%%%%%%%%%%%%%%%%%%%%%%%%%%%%%%%%%%%%%%%%%%%%%%%%%%%%%%%%%%%%%%%%%%%%%%%%%%%
	
	Continuous-time linear control systems on $\mathbb{R}^d$ are given by a family of differential equations of the form
	\begin{equation*}
		\Sigma_c:\ \dot{x}(t) = Ax(t) + Bu(t),\ t\in\mathbb{R},
	\end{equation*}
	where the control functions $u$ is in $\mathcal{U} = \{u:\mathbb{R}\rightarrow U\subset\mathbb{R}^m;\ u \mbox{ is piecewise continuous}\}$,  $A\in\mathbb{R}^{d\times d}$ and $B\in\mathbb{R}^{d\times m}$. This class of systems constitutes a widely known and applicable class of control systems (see Ogata \cite{OGA}). Discrete-time control systems are studied both for theoretical and practical reasons. Understanding this class of control is essential for engineers and researchers working in this field, where precise control and signal processing are required in digital environments. From an application point of view, it covers a wide range of digital applications, including communication systems, audio processing, automotive control systems, and aerospace applications. For instance, to prevent aliasing, i, e., when signals turn out indistinguishable through digital samples, anti-aliasing filters are often used through discrete-time control systems, see Rachid,  Pamarti and Daneshrad \cite{rpd}. The study of controllability of this system can be found in Sontag \cite{sontag1} where it is proved a classical result that establishes that the necessary and sufficient conditions for controllability: $\mbox{rank}[B\ AB\ \cdots A^{d-1}B]=d$ and $A$ admits only eigenvalues with zero real part .
	
	However, for digital signal processing modelling, for example, the discrete-time version of the above system becomes more effective (see Wilsky \cite{WILL79}). This class of systems is given by
	\begin{equation*}
		\Sigma_d:\ x_{k+1} = Ax_k + Bu_k,\ k\in\mathbb{N}_0, u=(u_k)_{k\in\mathbb{N}_0}\in\mathcal{U}:=U^{\mathbb{N}_0}
	\end{equation*}
	where $U\subset\mathbb{R}^m$ is non-empty, $A\in\text{Gl}(d,\mathbb{R})$ and $B\in\mathbb{R}^{d\times d}$.
	
	Regarding the controllability of ($\Sigma_d$), Colonius et al in \cite{CCS} proved that such a system is controllable if, and only if, $\mbox{rank}[B\ AB\ \cdots A^{d-1}B]=d$ and all eigenvalues of $A$ have absolute value equal to $1$. Moreover they characterized the control sets, that is, proper subsets of the state space where a similar controllability property holds. 
	
	A natural extension of the above system $(\Sigma_c)$ is known as a linear system on Lie group $G$, that is, a family of differential equations $\dot{g}(t) = \mathcal{X}(g(t)) + u X(g(t))$, where $\mathcal{X}$ comes from the automorphism of $G$, $X$ is a right invariant vector field  and $u \in \mathcal{U}$. This class of systems has been extensively studied and Lie theory has been shown to be powerful in obtaining results regarding the controllability, conjugacy and invariance entropy of this system (e.g. Ayala and Silva \cite{AyalaeAdriano1, AyalaeAdriano}, Ayala, Silva and Zsigmond \cite{ADSZ20}, Silva \cite{Adriano, ADS14}, Jouan \cite{Jouan1} and Jouan and Dath \cite{Jouan}).
	
	The discrete-time version of the linear systems on Lie groups was introduced by Colonius et al  \cite{CCS1}. Essentially, its dynamics are given by the following family of difference equations
	\[
	\Sigma\ :\ g_{k+1}=f(g_{k},u_{k}),\ k\in\mathbb{N}_0,\ u_{k}\in U,
	\]
	on a connected Lie group $G$, where $0\in U\subset\mathbb{R}^m$, $f_{0}:=f(\cdot,0)$ is
	an automorphism of $G$ and for each $u\in U$, $f_{u}:=f(\cdot, u):G\rightarrow G$ satisfies
	\[
	f_{u}(g)=f_{u}(e)\cdot f_{0}(g)\text{ for all }g\in G,
	\]
	where \textquotedblleft$\ \cdot$\textquotedblright denotes the product of $G$. The authors established a formula for calculating the outer invariance entropy of admissible pairs in terms of  eigenvalues of the Lie algebra automorphism $(df_0)_e:\mathfrak{g}\rightarrow\mathfrak{g}$.
	
	In this context, the objective of our study is to investigate the controllability of the system $(\Sigma)$. Our results were inspired by the findings of Silva \cite{Adriano}, who considered  continuous-time linear systems. Specifically, we demonstrate that if all eigenvalues of the automorphism $d(f_0)_e$ associated with a discrete-time linear system on a solvable Lie group have norm 1, and both the reachable and controllable sets from the identity are open, then the system is controllable. Additionally, it was proven that the converse of this result holds when $G$ is nilpotent.
	
	This paper is structured as follows: Section \ref{section2}   presents Lie-theoretic notations, facts, and some general properties of discrete-time linear control systems. In Section \ref{section3}, we provide   sufficient conditions for the controllability of discrete-time linear systems on connected solvable Lie groups. Furthermore, we constructed a class of linear systems on the two-dimensional Lie group $\hbox{Aff}(2,\R)$ and derived a condition for controllability. Section \ref{section4} focuses on nilpotent Lie groups, where we prove that the sufficient condition for controllability obtained for solvable Lie groups is also   necessary.

	\section{Preliminaries}\label{section2}
	
	This section establishes the notations, basic concepts and necessary results for the development of this study. In addition,  we present the definition of discrete-time linear control systems on Lie groups based on their properties.
	
	\subsection{Decompositions of Lie algebras and Lie groups}
	
	In this subsection we show some dynamical decomposition on Lie algebras and connected Lie groups introduced by Ayala, Román-Flores and Da Silva \cite{AyalaeRomaneAdriano} for a given automorphism.
	
	For a Lie algebra $\mathfrak{g}$ defined over a closed field, the generalized eigenspaces of an automorphism $\xi:\mathfrak{g}\rightarrow\mathfrak{g}$ associated with an eigenvalue $\alpha$ are given by
	\begin{equation*}
		\mathfrak{g}_{\alpha} = \{X \in \mathfrak{g}: (\xi - \alpha)^n X = 0, \hbox{ for some }n \in \N\}.
	\end{equation*}
	Due to \cite[Proposition 2.1]{AyalaeRomaneAdriano}, the following subspaces 
	\begin{equation}\label{liesubalgebras}
		\mathfrak{g}^+ = \bigoplus_{|\alpha|>1} \mathfrak{g}_{\alpha},\ \mathfrak{g}^- =\bigoplus_{|\alpha|<1} \mathfrak{g}_{\alpha},\  \mathfrak{g}^0 = \bigoplus_{|\alpha|=1} \mathfrak{g}_{\alpha}
	\end{equation}
	are Lie subalgebras of $\mathfrak{g}$, called \textit{unstable, stable} and \textit{center} Lie subalgebras of $\mathfrak{g}$ in relation to $\xi$, respectively. Moreover, the decomposition
	\begin{equation}\label{decomposition}
		\mathfrak{g} = \mathfrak{g}^+ \oplus \mathfrak{g}^0 \oplus \mathfrak{g}^-
	\end{equation}
	holds. It can  also be defined as \textit{center-unstable} and \textit{center-stable} Lie subalgebras%
	\begin{eqnarray}\label{centerdecomp}
		\mathfrak{g}^{+,0}=\mathfrak{g}^{+}\oplus\mathfrak{g}^{0}\ \text{ and
		}\ \mathfrak{g}^{-,0}=\mathfrak{g}^{-}\oplus\mathfrak{g}^{0},
	\end{eqnarray}
	resp. 
	
	If $\mathfrak{g}$ is a real Lie algebra and $\xi: \mathfrak{g} \longrightarrow \mathfrak{g}$ is an Lie endomorphism, consider $\overline{\mathfrak{g}}$ the complexification of $\mathfrak{g}$, as the construction in \cite[Chapter 12, Section 12.1]{sanmartin2}. Any endomorphism of a real Lie algebras $\mathfrak{g}$ induces a Lie endomorphism $\bar{\xi}: \overline{\mathfrak{g}} \rightarrow \overline{\mathfrak{g}}$ on the complex Lie algebra $\overline{\mathfrak{g}}$ \cite[Page 1479]{AyalaeRomaneAdriano}. If $n$ is the dimension of $\overline{\mathfrak{g}}$, we can define the subalgebras 
	\begin{equation*}
		\overline{\mathfrak{g}}_{\overline{\xi}} = \bigoplus_{\alpha \neq 0} \mathfrak{g}_{\alpha} \hbox{ and }\overline{\mathfrak{k}}_{\overline{\xi}} = \hbox{ker}(\bar{\xi}^n). 
	\end{equation*}
	and also $\overline{\mathfrak{g}}^0,\overline{\mathfrak{g}}^+$ and $\overline{\mathfrak{g}}^-$.
	
	Moreover, because all the above mentioned $\bar{\xi}$-subalgebras are invariant by complex conjugation, they are also the complexification of the following $\xi$-invariant subalgebras of $\mathfrak{g}$
	\begin{equation*}
		\mathfrak{g}_{\xi} =  \overline{\mathfrak{g}}_{\overline{\xi}} \cap \mathfrak{g}, \mathfrak{k}_{\xi} =  \overline{\mathfrak{k}}_{\overline{\xi}} \cap \mathfrak{g} \hbox{ and } \mathfrak{g}^* =  \overline{\mathfrak{g}}^*\cap \mathfrak{g}, *=0,+,-.    
	\end{equation*}
	with $\mathfrak{g}^+$ and $\mathfrak{g}^-$ nilpotent Lie subalgebras of $\mathfrak{g}$. 
	
	If $G$ is a connected Lie group with Lie algebra $\mathfrak{g}$ and $\phi$ is an automorphism of $G$, then $d\phi_e$ is an automorphism of $\mathfrak{g}$. Considering the subalgebras (\ref{liesubalgebras}) and (\ref{centerdecomp}) of $\mathfrak{g}$ in relation to $d\phi_e$, we denote by $G^{+}$, $G^{-}$, $G^0$, $G^{+,0}$ and $G^{-,0}$ the connected Lie subgroups of $G$ corresponding to $\mathfrak{g}^{+}$, $\mathfrak{g}^{-}$, $\mathfrak{g}^0$, $\mathfrak{g}^{+,0}$ and $\mathfrak{g}^{-,0}$, respectively. By considering a neighborhood $U$ of $0 \in \R^m$ and a function $f: U \times G \longrightarrow G$ such that $f(0,\cdot) = f_0: G \longrightarrow G$ is an automorphism, we can consider the control system
	\begin{equation*}
		\Sigma: g_{k+1} = f(u_k,x_k), k \in \Z^+,
	\end{equation*}
	and refer to the subgroups above as \textit{unstable}, \textit{stable}, \textit{center}, \textit{center-unstable}, and the \textit{center-stable} subgroups of the system $(\Sigma),$ respectively.
	
	\begin{remark}
		\begin{itemize}
			\item[]
			\item[1)] From \cite[Proposition 2.1]{AyalaeRomaneAdriano} it follows that  stable and unstable subalgebras are nilpotent. Moreover, because $[\mathfrak{g}^+,\mathfrak{g}^0]\subset\mathfrak{g}^{+}$, then $\mathfrak{g}^{+}$ is an ideal of $\mathfrak{g}^{+,0}$. Consequently, $G^+$ is a normal subgroup of $G^{+,0}$. The same holds for $\mathfrak{g}^{-}\subset\mathfrak{g}^{-,0}$ and $G^-\subset G^{-,0}$.
			\item[2)] Item 1 above implies that $G^{+,0}=G^+G^0=G^0G^+$ and $G^{-,0}=G^-G^0=G^0G^-$.
		\end{itemize}
	\end{remark}
	
	Given a homomorphism $\xi:\mathfrak{g}\rightarrow\mathfrak{g}$, a Lie subalgebra is $\xi$-invariant if $\xi(\mathfrak{h})\subset\mathfrak{h}$. When $\xi$ is an automorphism and $\mathfrak{h}$ is $\xi$-invariant, it is clear that $\xi^k(\mathfrak{h})=\mathfrak{h}$, for all $k\in\Z$. 
	
	At the Lie group level, given a homomorphism $\phi$ of $G$, we say that a Lie subgroup $H$ of $G$ is $\phi$-invariant if $\phi(H) \subset H$. Analogously, if $\phi$ is an automorphism and $H$ is $\phi$-invariant, it also holds that $\phi^{k}(H)=H$, for any $k \in \Z$. Note also that if $H$ is connected and $\mathfrak{h}$ is its Lie algebra, $H$ is $\phi$-invariant iff $\mathfrak{h}$ is $d\phi_e$-invariant, where $e$ denotes the identity of $G$.
	
	The next lemma shows that the above decompositions are preserved by a surjective Lie algebra homomorphism since this homomorphism commutes with two automorphisms.  
	
	\begin{lemma}\label{invdyngroup}
		Let $\eta:\mathfrak{g}\rightarrow\mathfrak{h}$ be a surjective Lie algebra homomorphism, $\xi_1$ and $\xi_2$ automorphisms of $\mathfrak{g}$ and $\mathfrak{h}$, respectively. If $\eta\circ\xi_1=\xi_2\circ\eta$, then
		$$\eta(\mathfrak{g}^+)=\mathfrak{h}^+,\ \eta(\mathfrak{g}^-)=\mathfrak{h}^- \mbox{ and } \eta(\mathfrak{g}^0)=\mathfrak{h}^0.$$
		
		In addition, if $G$ and $H$ are connected Lie groups associated with $\mathfrak{g}$ and $\mathfrak{h}$, respectively, and $\pi:G\rightarrow H$ is a surjective homomorphism such that $(d\pi)_e\circ\xi_1=\xi_2\circ(d\pi)_e$, then
		$$\pi(G^+)=H^+,\ \pi(G^-)=H^- \mbox{ and } \pi(G^0)=H^0.$$
	\end{lemma}
	
	\begin{proof}For an eigenvalue $\alpha$ of $\xi_1$, there exists $n\in\mathbb{N}$ with $(\xi_1-\alpha)^nX=0$, for all $X\in\mathfrak{g}$. From hypothesis,
		$$(\xi_2-\alpha)^n(d\pi)_e(X)=(d\pi)_e(\xi_1-\alpha)^n(X)=0.$$
		
		Since $\eta$ is surjective, $\alpha$ is also an eigenvalue of $\xi_2$, hence $\eta(\mathfrak{g}_\alpha)\subset\mathfrak{h}_\alpha$ which shows that $\eta(\mathfrak{g}^+)\subset\mathfrak{h}^+$, $\eta(\mathfrak{g}^-)\subset\mathfrak{h}^-$ and $\eta(\mathfrak{g}^0)\subset\mathfrak{h}^0$. The equality follows from the subjectivity of $\eta$ and the decomposition (\ref{decomposition}) applied to $\mathfrak{g}$ and $\mathfrak{h}$.
		
		The equalities at group level holds due subjectivity of $\pi$ and the equality $\pi(\exp_G(X))=\exp_H((d\pi)_eX)$.
	\end{proof}
	
	In sequence, it is established that solvable groups can be decomposed as a product of subgroups denoted as $G^{+}$, $G^0$, and $G^{-}$.

	\begin{proposition}\label{decomposable}
		If $G$ is a connected solvable Lie group with Lie algebra $\mathfrak{g}$ and $\phi$ an automorphism of $G$, then the Lie subgroups associated with decomposition (\ref{decomposition}) of $(d\phi)_e$ satisfy $G=G^{+,0}G^-=G^{-,0}G^+$.
	\end{proposition}
	
	\begin{proof} To prove that $G=G^{+,0}G^-$, we proceed by induction on $\dim G$. If $G$ is unidimensional, the group is abelian and the result follows. 
		
		Now, suppose the result is true for all connected solvable Lie groups with dimensions less than $d$. Let $G$ be a Lie group with $\dim G=d$ and consider the derivative series of its Lie algebra $\mathfrak{g}$
		$$\mathfrak{g}=\mathfrak{g}^{(0)}\supset\mathfrak{g}^{(1)}\supset\cdots\supset\mathfrak{g}^{(k)}\supset\mathfrak{g}^{(k+1)}=\{0\},$$
		where $\mathfrak{g}^{(i)}=[\mathfrak{g}^{(i-1)},\mathfrak{g}^{(i-1)}]$ for $i\in\{1,\ldots,k\}$. Each $\mathfrak{g}^{(i)}$ is an ideal of $\mathfrak{g}$, hence each Lie subgroup $G^{(i)}$ associated with $\mathfrak{g}^{(i)}$ is normal in $G$. The $d\phi_e$-invariance of $\mathfrak{g}^{(i)}$ implies the $\phi$-invariance of $G^{(i)}$. In particular, $G^{(k)}$ is $\phi$-invariant, abelian and normal; therefore its closure $\overline{G^{(k)}}$ is a closed Lie subgroup with the same properties. The Lie group $H:=G/\overline{G^{(k)}}$ is solvable with
		$$\dim H=\dim G-\dim\overline{G^{(k)}}\leq\dim G-\dim G^{(k)}<\dim G,$$
		because $\dim G^{(k)}>0$. Denote by $\pi$ the natural projection of $G$ on $H$ and note that
		$$H=H^{+,0}H^-=\pi(G^{+,0})\pi(G^-)=\pi(G^{+,0}G^-)$$
		by Lemma \ref{invdyngroup} and induction hypothesis. Then $G=G^{+,0}G^-\overline{G^{(k)}}=G^{+,0}\overline{G^{(k)}}G^-$. 
		
		If $\overline{\mathfrak{g}^{(k)}}$ denotes the Lie algebra of $\overline{G^{(k)}}$, one has that $\overline{\mathfrak{g}^{(k)}}^{+,0}=\overline{\mathfrak{g}^{(k)}}\cap\mathfrak{g}^{+,0}$ and $\overline{\mathfrak{g}^{(k)}}^{-}=\overline{\mathfrak{g}^{(k)}}\cap\mathfrak{g}^{-}$ by $(d\phi)_e$-invariance of $\overline{\mathfrak{g}^{(k)}}$. Moreover, since it is abelian, it holds that $\overline{G^{(k)}}=\overline{G^{(k)}}^{+,0}\overline{G^{(k)}}^-$, which shows that $\overline{G^{(k)}}^{+,0}\subset G^{+,0}$ and $\overline{G^{(k)}}^-\subset G^-$, therefore
		$$G=G^{+,0}\overline{G^{(k)}}G^-=G^{+,0}\overline{G^{(k)}}^{+,0}\overline{G^{(k)}}^-G^-\subset G^{+,0}G^-\subset G,$$
		and the statement holds. The equality $G=G^{-,0}G^+$ follows analogously.
	\end{proof}
	
	\begin{proposition}\label{compactsubgroup}
		Every compact, connected and $\phi$-invariant Lie subgroup of $G$ is contained in $G^0$.
	\end{proposition}
	
	\begin{proof}
		Let $H$ be a Lie subgroup of $G$ that satisfies  the above conditions. If $\mathfrak{h}$ denotes the Lie algebra of $H$, then \cite[Corollary 4.25]{Knapp} implies that $\mathfrak{g}=\mathfrak{z}(\mathfrak{h})\oplus[\mathfrak{h},\mathfrak{h}]$, where $\mathfrak{z}(\mathfrak{h})$ is the center of $\mathfrak{h}$ and $[\mathfrak{h},\mathfrak{h}]$ is semisimple. Note that the $\phi$-invariance of $H$ implies a $d\phi_e$ invariance of $\mathfrak{h}$.
		
		The connected Lie group associated with $\mathfrak{z}(\mathfrak{h})$ is the connected component of $Z(G)$, which is denoted by $Z(G)_0$. Since $Z(G)_0$ is compact and abelian, the subset $S \subset Z(G)_0$ of all elements with a finite order is dense in $Z(G)_0$. In this case, given $X\in\mathfrak{z}(\mathfrak{h})$ such that $\exp X\in S$, there is $k\in\mathbb{N}$ with $(\exp X)^k=e$, hence
		\begin{eqnarray*}
			e=\phi((\exp X)^k)=(\phi(\exp X))^k=\exp (d\varphi)_e^kX.
		\end{eqnarray*}
		Since the exponential map of an abelian Lie group is a diffeomorphism and $(d\varphi)_e^kX=0$, then $X=0$. Hence $\mathfrak{z}(\mathfrak{h})$ is trivial.
		
		In addition, owing to the compactness of $H$ and the semisimplicity of $[\mathfrak{h},\mathfrak{h}]$ it follows that the Cartan-Killing form restricted to $[\mathfrak{h},\mathfrak{h}]$ is non-degenerated and negative defined (see \cite[Theorem 1.45 and Corollary 4.26]{Knapp}). Since $d\phi_e$ is an automorphism, it is an isometry by Cartan-Killing form, hence all eigenvalues of $d\phi_e|_{[\mathfrak{h},\mathfrak{h}]}$ has absolute value equal to $1$, which yields $\mathfrak{h}=[\mathfrak{h},\mathfrak{h}]\subset\mathfrak{g}^0$.
	\end{proof}
	
	By the proposition \ref{compactsubgroup} the following result is immediate. 
	
	\begin{corollary}
		If $\phi$ is an automorphism of a compact Lie group $G$, then $d\phi_e$ has only eigenvalues with an absolute value equal to $1$.
	\end{corollary}
	
	The next lemma will be frequently used and can be found in  \cite[Lemma 3.1]{sontag1}.
	
	\begin{lemma}\label{lema212}Let $G$ be a Lie group with Lie algebra $\mathfrak{g}$ and $N$ be a normal Lie subgroup of $G$ with Lie algebra $\mathfrak{n}$. Then for every $X \in \mathfrak{g}$, we have  
		\begin{equation*}
			\exp{(X + \mathfrak{n})}\subset \exp{(X)}N. 
		\end{equation*}
	\end{lemma}

	\subsection{Linear control systems on Lie groups}
	This subsection presents some general properties of the discrete-time control systems. We start by recalling that a discrete-time control system on a topological space $M$ is given by the difference equation
	\begin{equation}\label{sistgeneral}
		x_{k+1} = f(x_k, u_k),\ k \in \mathbb{N}_0:=\mathbb{N}\cup\{0\},\ u_k \in U,
	\end{equation}
	where $U$ is a non-empty set and $f:M\times U\rightarrow M$ is a map such that for each $u\in U$, $f_u(\cdot):=f(\cdot, u):M\rightarrow M$ is continuous function. Set $M$ is called the \textit{state space} of (\ref{sistgeneral}) and $U$ is the \textit{control range}. Moreover, the \textit{shift space} is defined as the set of all sequences in $U$, that is, $\mathcal{U}:=\displaystyle\prod_{i=0}^{\infty}U$. The elements in $\mathcal{U}$ are called \textit{controls}. In this case, the following solutions to (\ref{sistgeneral}) are well-defined
	\begin{equation*}
		\varphi(k,x_0,u) = \left\{
		\begin{array}
			[c]{lll}%
			x_0, & \mbox{if} & k=0\\
			f_{u_{k-1}}\circ\cdots\circ f_{u_{1}}\circ f_{u_{0}}(x_0), & \mbox{if} & k\geq1
		\end{array},
		\right.
	\end{equation*}
	for all $k\in\mathbb{N}_0$, $x_0\in M$ and $u=(u_i)_{i\in\mathbb{N}_0}\in\mathcal{U}$. Note that $\varphi(k,\cdot, u):M\rightarrow M$ is continuous for each $k\in\mathbb{N}_0$ and $u=(u_i)_{i\in\mathbb{N}_0}\in\mathcal{U}$.
	
	The \textit{shift} map $\Theta: \N_0\times \mathcal{U} \rightarrow \mathcal{U}$ given by $\Theta(k, (u_j))=\Theta_k((u_j)) := (u_{j+k})$ defines a dynamical system on $\mathcal{U}$. It is well known that $\varphi$ satisfies the cocycle property, that is, 
	$$\varphi(k + l, x, u) = \varphi(k, \varphi(l,x,u), \Theta_l(u)),\ \forall\ k,l \in \N_0, \ \forall\ x\in M, \ \forall\ u\in\mathcal{U}.$$
	
	\begin{definition}For $x \in M$, the reachable and   controllable sets from $\textbf{x}$ at time $k$ are given by
		\begin{eqnarray*}
			\mathcal{R}_k(x) = \{y \in M: \hbox{ there is }u \in \mathcal{U} \hbox{ with }\varphi(k,x,u)=  y\}.
		\end{eqnarray*}
		and
		\begin{eqnarray*}
			\mathcal{C}_k(x) = \{y \in M: \hbox{ there is }u \in \mathcal{U} \hbox{ with }\varphi(k,y,u)= x\},
		\end{eqnarray*}
		respectively. Moreover, the reachable and  controllable sets from $\textit{x}$ \textit{up to time} $\textit{k}$ are given by $\mathcal{R}_{\leq k}(x) = \displaystyle\bigcup_{t\leq k} \mathcal{R}_t(x)$ and $\mathcal{C}(x) = \displaystyle\bigcup_{t\leq k} \mathcal{C}_t(x)$, respectively. The sets $\mathcal{R}(x) = \displaystyle\bigcup_{k \in \mathbb{N}} \mathcal{R}_k(x)$ and $\mathcal{C}(x) = \displaystyle\bigcup_{k \in \mathbb{N}} \mathcal{C}_k(x)$ denote the reachable and   controllable sets from $x$, respectively.  
	\end{definition}
	
	System (\ref{sistgeneral}) is \textit{forward accessible} (resp. \textit{backward accessible}) if $\hbox{int}\mathcal{R}(x) \neq \emptyset$ (resp. $\hbox{int}\mathcal{C}(x) \neq \emptyset$), for all $x \in M$ and it is called \textit{accessible} if both conditions are satisfied. 
	
	According to Wirth \cite{Wir98}, for any $k \in \N$ and $g \in G$, the smooth map $G_k: G \times U^k \rightarrow G$ is defined as follows:
	\begin{equation*}
		G_k(g,u) = \varphi(k,g,u).
	\end{equation*}
	
	A pair $(g,u) \in G \times \hbox{int}U^k$ such that  $\hbox{rank}\left[\frac{\partial}{\partial u}G_k(g,u)\right] = \dim G$ is called \textit{regular}. The \textit{regular reachable set at time} $\textit{k}$ is defined by
	\begin{equation*}
		\hat{\mathcal{R}}_k(g) = \{\varphi(k,g,u): (x,u) \hbox{ is regular}\}
	\end{equation*}
	and the \textit{regular reachable set} as $\hat{\mathcal{R}}(g)  =\cup_{k \in \N} \hat{\mathcal{R}}_k(g)$. We can prove that $\hat{\mathcal{R}}(g)$ is open, for every $g \in G$.
	
	Now, let us define the main object of this work. 
	
	\begin{definition}\label{deflincontsys}
		Consider a discrete-time control system
		\begin{equation}\label{linearcontrolsystem}
			g_{k+1} = f(g_k, u_k), u_k \in U,
		\end{equation}
		on a connected Lie group $G$ with $U\subset\R^m$ a compact neighborhood of $0$. System (\ref{linearcontrolsystem}) is called \textbf{linear} if $f_0$ is an automorphism of $G$ and for each $g\in G$
		\begin{equation}\label{transl}
			f_u(g) = f_u(e) \cdot f_0(g).
		\end{equation}
		where $"\cdot"$ denotes the product on $G$.
	\end{definition}
	
	The group product will be omitted where context makes it evident. Moreover,  equation (\ref{transl}) can be written as:
	\begin{equation*}
		f_u(g) = f_u(e)f_0(g) = L_{f_u(e)}\circ f_0(g),
	\end{equation*}
	where $L_{g}$ is the left translation by $g\in G$. Considering the above expression, we can see that $f_u$ is a diffeomorphism of $G$, for each $u\in U$, with inverse 
	\begin{equation*}
		f_u^{-1}(g) = f_0^{-1}\circ L_{(f_u(e))^{-1}}(g)= f_0^{-1}((f_u(e))^{-1} \cdot g). 
	\end{equation*}
	
	\begin{remark}
		It follows from Definition \ref{deflincontsys} that:
		\begin{itemize}
			\item[1)] $\mathcal{U}$ endowed with the product topology is a compact space.
			\item[2)] for all $k\in\mathbb{N}_0$ and $u\in\mathcal{U}$, the map $\varphi(k,\cdot,u)$ is a diffeomorphism of $G$.
		\end{itemize}
	\end{remark}
	
	\begin{example}
		\label{Example2}The standard example of discrete-time linear system is the difference equation
		\[
		x_{k+1}=Ax_{k}+Bu_{k},\ \ u_{k}\in U,
		\]
		where $G$ is the additive euclidean space $\R^d$, $A\in GL(d,\mathbb{R}),B\in\mathbb{R}^{d\times m}$, and $0\in
		U\subset\mathbb{R}^{m}$. Map $f:\mathbb{R}^{d}\times U\rightarrow
		\mathbb{R}^{d}$ is given by $f(x,u)=Ax+Bu$ and satisfies 
		\begin{itemize}
			\item $f_{0}(x)=Ax$ is an automorphism of $\mathbb{R}^{d}$;
			\item $f_{u}(x)=Ax+Bu=f_{0}(x)+f_{u}(0)=f_{u}(0)+f_{0}(x)$, since $f_u(0) = Bu$
		\end{itemize}
		In this case, the solutions are given by
		\[
		\varphi(k,x,u)=A^{k}x+\sum_{j=0}^{k-1}A^{k-1-j}Bu_{j}.
		\]
	\end{example}
	
	The next proposition shows that the solutions of \ref{linearcontrolsystem}, from an element $g\in G$ can be given by a translation of the solution from $e\in G$ (see \cite[Proposition 3]{CCS1})
	
	\begin{proposition}\label{prop52} 
		Consider the discrete-time linear system (\ref{linearcontrolsystem}) on a Lie group $G$. Then, for all $g \in G$ and $u\in \mathcal{U}$ it holds that
		\begin{equation*}
			\varphi(k,g,u) = \varphi(k,e,u)f_0^k(g). 
		\end{equation*}
	\end{proposition}
	
	In the case of linear systems, the identity $e$ of $G$ satisfies $e \in \mathcal{R}_k(e)$ for all $k \in \N$, because it is a fixed point of $f_0$. In addition, using the notation $\mathcal{R}(e) = \mathcal{R}$, $\mathcal{R}_k(e) = \mathcal{R}_k$ and $\mathcal{R}_{\leq k} = \mathcal{R}_{\leq k}(e)$, we obtain the following proposition. 
	
	\begin{proposition}\label{reachablesetprop}
		For all $k, k_1, k_2\in\mathbb{N}$ $g\in G$ and $u\in\mathcal{U}$, it holds that 
		\begin{itemize}
			\item[1)] $\mathcal{R}_{k} = \mathcal{R}_{\leq k}$;
			\item[2)] If $k_1\leq k_2$, then $\mathcal{R}_{k_1} \subset \mathcal{R}_{k_2}$; 
			\item[3)] $\mathcal{R}_{k}(g) = \mathcal{R}_{k} f_0^{k}(g)$;
			\item[4)] If $k_1, k_2 \in \N$, then $\mathcal{R}_{k_1 + k_2} = \mathcal{R}_{k_1} f_0^{k_1}(\mathcal{R}_{k_2}) = \mathcal{R}_{k_2} f_0^{k_2}(\mathcal{R}_{k_1})$;
			\item[5)] For any $u \in \mathcal{U}$, $g \in G$ and $k \in \N$, then 
			\begin{equation*}
				\varphi(k,\mathcal{R}(g),u) \subset \mathcal{R}(g);
			\end{equation*}
			\item[6)] $e \in \hbox{int}\mathcal{R}$ if and only if $\mathcal{R}$ is open. 
		\end{itemize}
	\end{proposition}
	
	\begin{proof}
		\begin{itemize}
			\item[1)] It is clear that $\mathcal{R}_{k} \subset \mathcal{R}_{\leq k}.$ On the other hand, taking $t \in [1,k) \cap \N$, an arbitrary $u \in \mathcal{U}$ and considering the control
			\begin{equation*}
				v=\left\{
				\begin{array}{lcl}
					0,& \text{for}& j < k - t\\
					u_{j - k + t},& \text{for}& j \geq k -t
				\end{array},
				\right.
			\end{equation*}
			we have
			\begin{equation*}
				\varphi(t,e,u) = \varphi(t,\varphi(k - t,e,0),u) = \varphi(t,\varphi(k-t,e,v),\Theta_{k-t}(v))=\varphi(k,e,v), 
			\end{equation*}
			which implies that $\mathcal{R}_k \subset \mathcal{R}_{t}$, proving the statement.
			
			\item[2)] Consequence of (1). 
			
			\item[3)] It follows from Proposition \ref{prop52}.
			\item[4)] Note that
			\begin{eqnarray*}
				\varphi(k_1 + k_2,e,u) &=& \varphi(k_1, \varphi(k_2,e,u),\Theta_{k_2}(u))\\
				&=&\varphi(k_1, e, \Theta_{k_2}(u)) f_0^{k_1}(\varphi(k_2,e,u)) \in \mathcal{R}_{k_1}f_0^{k_1}(\mathcal{R}_{k_2})
			\end{eqnarray*}
			which means that $\mathcal{R}_{k_1 + k_2} \subset \mathcal{R}_{k_1}f_0^{k_1}(\mathcal{R}_{k_2})$. Now, given $u,v \in \mathcal{U}$, we have
			\begin{eqnarray*}
				\varphi(k_1, e, u)f_0^{k_1}(\varphi(k_2,e,v))&=& \varphi(k_1, \varphi(k_2,e,v),u)\\
				&=&\varphi(k_1 + k_2, e,w), 
			\end{eqnarray*}
			with 
			\begin{equation*}
				w = 
				\left\{
				\begin{array}{lcc}
					v_j,& j < k_2\\
					u_{j - k_2},& j \geq k_2
				\end{array}.
				\right.
			\end{equation*}
			
			The inclusion $\mathcal{R}_{k_2}f_0^{k_2}(\mathcal{R}_{k_1})\subset \mathcal{R}_{k_1 + k_2}$ follows by the same arguments above. 
			
			\item[5)] Take $k \in \N$, $g \in G$ and $u \in \mathcal{U}$. If $h \in \mathcal{R}(g)$, then $h = \varphi(t,g,v)$ for some $t \in \N$ and $v \in \mathcal{U}$. Thus 
			\begin{equation*}
				\varphi(k,h,u) = \varphi(k,\varphi(t,g,v),u) = \varphi(k + t, g,w) \in \mathcal{R}(g),
			\end{equation*}
			with 
			\begin{equation*}
				w = \left\{
				\begin{array}{lcc}
					v_j,& j < t\\
					u_{j - t},& j \geq t
				\end{array}.
				\right.
			\end{equation*}
			
			\item[6)] As $e \in \mathcal{R}$ if $\mathcal{R}$ is open then $e \in \hbox{int}\mathcal{R}$. Suppose that $e \in \hbox{int}\mathcal{R}$ and take $g \in \mathcal{R}$. Then there are $t \in \N$ and $u \in \mathcal{U}$ such that 
			\begin{equation*}
				\varphi(t,e,u)  = g. 
			\end{equation*}
			
			Take $V_e \subset \mathcal{R}$ as the neighborhood of $e$. As $f_u$ is diffeomorphism, the set $V_g = \varphi(t,V_e, u)$ is a neighborhood of $g$ and $\varphi(t,V_e,u) \subset \varphi(t,\mathcal{R},u) \subset \mathcal{R}$.
		\end{itemize}
	\end{proof}
	
	Because the map $f_u$ of the linear system (\ref{linearcontrolsystem}) is a diffeomorphism of $G$ for each $u\in U$ we can define its \textit{reversed counterpart} as
	\begin{eqnarray}\label{inverselinearcontrolsystem}
		h_{k+1}=\tilde{f}_{u_k}(h_k), \ u\in\mathcal{U},
	\end{eqnarray}
	where $\tilde{f}_u(h)=f_u^{-1}(e)f_0^{-1}(h)$ for all $h\in G$. Note that (\ref{inverselinearcontrolsystem}) is also a linear system on $G$ with $\tilde{f}_0(x)=f_0^{-1}(x)$, hence $d\tilde{f}_0=df_0^{-1}$. In this case, if $\alpha$ is an eigenvalue of $df_0$, $\alpha^{-1}$ is an eigenvalue of $d\tilde{f}_0$.
	
	Therefore, we can consider the generalized eigenspaces
	\begin{equation*}
		\mathfrak{g}_*^+ = \sum_{|\alpha|>1} \mathfrak{g}^*_{\alpha},\ \mathfrak{g}_*^- = \sum_{|\alpha|<1} \mathfrak{g}^*_{\alpha},\ \mathfrak{g}_*^0 = \sum_{|\alpha|=1} \mathfrak{g}^*_{\alpha}.
	\end{equation*}
	where $\mathfrak{g}_{\alpha}^*$ is the generalized eigenspace associated with the eigenvalue $\alpha$ of $d\tilde{f}_0$. It can be seen that $\mathfrak{g}^-_* = \mathfrak{g}^+, \mathfrak{g}^+_* = \mathfrak{g}^-$ and $\mathfrak{g}^0_* = \mathfrak{g}^0$. Analogously, $G_*^+=G^-$, $G_*^-=G^+$ and $G_*^0=G^0$.
	
	\begin{example}
		Consider the system in example \ref{Example2}. In this case, the reverse counterpart of the system is given by 
		\begin{equation*}
			x_{k+1}=A^{-1}x_{k}-A^{-1}Bu_{k},\ \ u_{k}\in U\subset\mathbb{R}^{m}.
		\end{equation*}
	\end{example}
	
	Let us denote by $\mathcal{R}_k^*$ and $\mathcal{C}_k^*$ as the reachable and   controllable sets from $e$ up to time $k$ in (\ref{inverselinearcontrolsystem}), respectively. Then the following lemma holds.
	
	\begin{lemma}\label{setsinvsys}
		It holds that $\mathcal{R}_k^*=\mathcal{C}_k$ and $\mathcal{R}_k=\mathcal{C}_k^*$, for all $k\in\mathbb{N}$.
	\end{lemma}
	
	\begin{proof}
		Note that for any $k\in\mathbb{N}$, $g\in G$ and $u\in\mathcal{U}$, we have
		\begin{eqnarray*}
			\tilde{f}_{u_{k-1}}\circ\cdots\circ\tilde{f}_{u_{0}}(g)=e \Leftrightarrow \tilde{f}_{u_{0}}^{-1}\circ\cdots\circ\tilde{f}_{u_{k-1}}^{-1}(e)=g
			\Leftrightarrow f_{u_{0}}\circ\cdots\circ f_{u_{k-1}}(e)=g,
		\end{eqnarray*}
		because $\tilde{f}_u^{-1}=f_u$, for all $u\in U$. If $\varphi^*$ denotes the solution of (\ref{inverselinearcontrolsystem}), then
		$$\varphi^*(k,g,u)=e\Leftrightarrow\varphi(k,e,\tilde{u})=g,$$
		where
		$$
		\tilde{u}=\left\{
		\begin{array}{lcc}
			u_{k-1-j},& j <k\\
			0,& j \geq k
		\end{array},
		\right.
		$$
		This shows that $\mathcal{R}_k=\mathcal{C}_k^*$. The other equality follows analogously.
	\end{proof}
	
	In sequence we present a lemma that will be widely used throughout the work.
	
	\begin{lemma}\label{AginA} Let $g \in \mathcal{R}$ such that $f_0^{-k}(g) \in \mathcal{R}$, for all $k\in\Z$. Then $\mathcal{R}\cdot g \subset \mathcal{R}$. 
	\end{lemma}
	
	\begin{proof}
		In fact, given $k \in \N$ and $u \in \mathcal{U}$ one has
		\begin{eqnarray*}
			\varphi(k,e,u)g = \varphi(k,e,u)f_0^k(f_0^{-k}(g))=\varphi(k,f_0^{-k}(g),u),
		\end{eqnarray*}
		for all $g\in G$. Hence, if we assume that $g$ is an element in $\mathcal{R}$ satisfying the required assumptions, then there are $l\in\N_0$ and $v \in \mathcal{U}$ such that $f_0^{-k}(g) = \varphi(l,e,v)$. Then 
		\begin{equation*}
			\varphi(k,e,u)g=\varphi(k,f_0^{-k}(g),u)=\varphi(k,\varphi(l,e,v),u) = \varphi(k+l,e,w),
		\end{equation*}
		with 
		\begin{equation*}
			w = \left\{
			\begin{array}{lcc}
				v_j,& j < l\\
				u_{j - l},& j \geq l
			\end{array},
			\right.
		\end{equation*}
		that is $\mathcal{R}\cdot g \subset \mathcal{R}$.
	\end{proof}
	
	To simplify, we denote $(df_0)_e$ by $df_0$. Since $f_0$ and $df_0$ are automorphisms on $G$ and $\mathfrak{g}$, respectively, we have
	$$f_0^n(\exp X)=\exp (df_0^nX),$$
	for all $n\in\Z$ and all $X\in\mathfrak{g}$. This fact is useful in the next results, whose proof can be easily adapted for the discrete-time case from   Corollaries 3.2 and 3.3 in \cite{Adriano}. 
	
	\begin{corollary}Suppose that $H$ is a connected $f_0$-invariant Lie subgroup of $G$ with Lie algebra $\mathfrak{h}$. If $\exp{X} \in \mathcal{R}$, for any $X \in \mathfrak{h}$ then $H \subset \mathcal{R}$. 
	\end{corollary}
	
	\begin{proof}Corollary 3.2 from \cite{Adriano}.  
	\end{proof}
	
	\begin{corollary}\label{corBinv}Suppose that $H$ is a connected Lie subgroup of $G$ and there is a neighborhood $B$ of $e$ in $H \cap \mathcal{R}$ which is invariant by $f_0$ and $f_0^{-1}$, then $H$ is $f_0$-invariant and $H \subset \mathcal{R}$. 
	\end{corollary}
	
	\begin{proof}Corolalary 3.3 from \cite{Adriano}. 
	\end{proof}
	
	The next section contains the main results of this paper, including the condition for the controllability of discrete-time linear systems on the affine two-dimensional Lie group. 
	
	\section{Controllability on solvable Lie groups}\label{section3}
	
	From now on, we will consider the discrete-time linear system (\ref{linearcontrolsystem}) on a connected solvable n-dimensional  Lie group $G$. Also, we will consider $f_0$ satisfying $df_0(X) = e^{\mathcal{D}_1} \circ  ... \circ e^{\mathcal{D}_k}(X)$, for $\mathcal{D}_j \in \hbox{Der}(\mathfrak{g})$. This section presents some controllability results and an application to linear systems on a non-abelian $2$-dimensional solvable Lie group.
	
	\subsection{General results}
	
	The aim of this subsection is to present a sufficient condition for the controllability of discrete-time linear systems on solvable Lie groups.
	
	At first, regarding to the automorphism $f_0$, we have the following result. 
	
	\begin{proposition}\label{propnilpotent}Let $\mathfrak{h} \subset \mathfrak{g}$ be a Lie subalgebra of $\mathfrak{g}$, $\mathfrak{n}$ a ideal of $\mathfrak{h}$ such that $\mathcal{D}_j(\mathfrak{h}) \subset \mathfrak{n}$. If $N = \langle \exp(\mathfrak{n}) \rangle \subset \mathcal{R}$, then $H = \langle \exp(\mathfrak{h}) \rangle \subset \mathcal{R}$.  
	\end{proposition}
	
	\begin{proof} Given $X \in \mathfrak{h}$, we have 
		\begin{equation*}
			e^{k\mathcal{D}_j}X = X + \sum_{n \geq 1} \frac{k^n \mathcal{D}_j^n}{n!} X = X + Y,
		\end{equation*}
		where $Y \in \mathfrak{n}$. As $df_0 = e^{\mathcal{D}_k} \circ  ...  \circ e^{\mathcal{ D}_k}$, using the expression above, we conclude that 
		\begin{equation*}
			df_0^l X = X + Z,  
		\end{equation*}
		where $Z \in \mathfrak{n}$. Now, as $\mathfrak{n}$ is an ideal, then $N$ is normal in $H$. By Lemma \ref{lema212}, $\exp(X + Z) = \exp(X)g'$ for some $g' \in N$. Therefore, the rest of proof follows using the same arguments as in Proposition 3.4 from \cite{Adriano}. 
	\end{proof}
	
	The \textit{Baker-Campbell-Hausdorff} is a series  $C(X,Y)$ defined on $\mathfrak{g}$ constructed as 
	\begin{equation}\label{BCH}
		C(X,Y) = X + Y + \sum_{n \geq 2} C_n(X,Y),
	\end{equation}
	where 
	\begin{equation*}
		\begin{array}{ccl}
			C_2(X,Y) &=& - \dfrac{1}{2}[X,Y], \\
			C_3(X,Y) &=& \dfrac{1}{12}\left([X,[X,Y]] - [Y,[Y,X]]\right),\\
			C_4(X,Y) &=& -\dfrac{1}{24}([Y,[X,[X,Y]]]),
		\end{array}
	\end{equation*}
	and so on. The other elements of $C_n$ are  difficult to express. In \cite[Proposition 8.5, Chapter 8]{sanmartin1}   proved that   expression \ref{BCH} is always valid under the condition $|X|, |Y| \leq \rho$, for a sufficiently small $\rho > 0$. 
	
	Now, consider $N \subset G^0$ as a nilpotent Lie subgroup of $G$. If $\mathfrak{n}$ is the Lie subalgebra of $N$, then there exists $l\in\mathbb{N}$ such that its lower central series satisfies
	\begin{equation*}
		\mathfrak{n}_1 \supset \mathfrak{n}_2 \supset \cdots \supset \mathfrak{n}_{l+1} = \{0\}.   
	\end{equation*}
	where $\mathfrak{n}_1 = \mathfrak{n}$ and $\mathfrak{n}_{i+1} = [\mathfrak{n}, \mathfrak{n}_i]$. Each set $\mathfrak{n}_i$ can be considered as 
	\begin{equation}\label{n_i}
		\mathfrak{n}_i = \bigoplus_{|\alpha| = 1} \mathfrak{n}_{i,\alpha},
	\end{equation}
	where $\mathfrak{n}_{i,\alpha} = \mathfrak{g}_{\alpha} \cap \mathfrak{n}_i$. We also can consider 
	\begin{equation*}
		\mathfrak{n}_{i, \alpha} = \bigcup_{j \in \N_0} \mathfrak{n}_{i, \alpha}^{j}, 
		\mathfrak{n}_{i, \alpha}^{j}= \{X \in \mathfrak{n}_{i, \alpha}: (df_0 - \alpha)^j X = 0\}.
	\end{equation*}
	
	In this section, it is assumed that $\mathcal{R}$ is open. The following lemma is the most important part in this section. 
	
	\begin{lemma}\label{G+0inreach} Let $N \subset G^0$ be a nilpotent connected $f_0$-invariant Lie subgroup of $G^0$. Then $N \subset \mathcal{R}.$  
	\end{lemma}
	
	\begin{proof}
		As $N$ is $f_0$-invariant, then $\mathfrak{n}$ is $df_0$-invariant, which implies that $df_0(\mathfrak{n})=\mathfrak{n}$. Consequently $df_0 (\mathfrak{n}_i)=\mathfrak{n}_i$, for any $i=1,\ldots,l+1.$ Let us denote by $N_i = \langle \exp{\mathfrak{n}_i} \rangle$ the connected subgroup associated with the ideal $\mathfrak{n}_i$, as in   (\ref{n_i}). Each $N_i$ is a normal subgroup of $N$ and $f_0$-invariant, since $f_0$ is an automorphism and each $\mathfrak{n}_i$ is ideal. 
		
		Note that if for every $X \in \mathfrak{n}_{i,\alpha}$ we have $\exp{X} \in \mathcal{R}$, then $N_i \in \mathcal{R}$. In fact, as $\mathfrak{n}_i = \bigoplus_{|\alpha|=1} \mathfrak{n}_{i,\alpha}$, the set 
		\begin{equation*}
			B = \prod_{|\alpha|=1} \exp{\mathfrak{n}_{i,\alpha}}
		\end{equation*}
		is a neighborhood of $e$ in $N_i$. In addition, knowing that $df_0^k(\mathfrak{n}_{i,\alpha}) = \mathfrak{n}_{i,\alpha}$ for every $k \in \Z$ we have 
		\begin{equation*}
			f_0^k(\exp{\mathfrak{n}_{i,\alpha}}) = \exp{(df_0^k(\mathfrak{n}_{i,\alpha}))} = \exp{\mathfrak{n}_{i, \alpha}},
		\end{equation*}
		for every $k\in \Z$. It is clear that $B$ is $f_0$ and $f_0^{-1}$-invariant and as $\exp{X} \in \mathcal{R}$ for every $X \in \mathfrak{n}_{i,\alpha}$, the Lemma \ref{AginA} ensures that $B \subset \mathcal{R}$. By Corollary \ref{corBinv} we get $N_i \subset \mathcal{R}$. 
		
		We claim that if $N_{i+1} \subset \mathcal{R}$ then $N_i \subset \mathcal{R}$. In fact, note that we only need to prove that $\exp{X} \in \mathcal{R}$, for every $X \in \mathfrak{n}_{i,\alpha}$ with $|\alpha| = 1$. Using   decomposition (\ref{n_i}), it is sufficient to prove that for every $j \in \N_0$ and $X \in \mathfrak{n}^j_{i,\alpha}$ we have $\exp{X} \in \mathcal{R}$. It is clear for $j=0$ since $\mathfrak{n}_{i, \alpha}^{0} = \{0\}$ and $e \in \hbox{int}\mathcal{R}$. If this is true for $j > 0$ then $\exp{Z} \in \mathcal{R}$ for every $Z \in \mathfrak{n}^{j-1}_{i, \alpha}$. Take $X \in \mathfrak{n}^{j}_{i, \alpha}$. Then there is an $m \in \N$ such that $v = \frac{X}{m} \in U$, where $\exp|_U$ is a diffeomorphism and $\exp(U) \subset \mathcal{R}$. Therefore, there is a $\tau \in \N$ such that $\exp{v} \in \mathcal{R}_{\tau}$. From Proposition \ref{reachablesetprop}, item 2, we have $\mathcal{R}_{\tau} \subset \mathcal{R}_{\tau n}$ for any $n \in \N$. Taking the polar form of a complex number, we have 
		\begin{equation*}
			\alpha^{\tau n} = \cos{(\tau n \theta)} + i \sin{(\tau n \theta)}. 
		\end{equation*}
		
		WLOG, let us assume that $\theta \in 2\pi \Q$. For any $n \in \N$, there exists an $N \in \Z$ such that $n\theta = 2N\pi$. Then $\alpha^{\tau n} = 1$.  In particular, we can take $n_{k_1}, n_{k_2} \in \N$ such that $v \in \mathcal{R}_{\tau(n_{k_2} - n_{k_1})}$. By taking $k \in \N$, we get 
		\begin{equation*}
			(df_0)^{k} v = (df_0 - \alpha + \alpha)^{k} v = \sum_{i=0}^{k} \binom{k}{i} (df_0 - \alpha)^{k - i} \alpha^i v =  \alpha^{k} v + v_k
		\end{equation*}
		with $\alpha^{k} v \in \mathfrak{n}_{i,\alpha}^j$ and 
		\begin{equation}\label{v_k}
			v_k =  \sum_{i=0}^{k -1} \binom{k}{i}(df_0 - \alpha)^{k   - i} \alpha^i v \in \mathfrak{n}_{j,\alpha}^{j-1}.
		\end{equation}
		for any $k \geq 1$. Moreover, we can consider 
		\begin{equation}\label{df0}
			\alpha^{k }v = df_0^{k } v + u_k, \forall k \in \N
		\end{equation}
		with $u_k = -v_k$. One can notice that $u_k \in \mathfrak{n}_{i,\alpha}^{j-1}$. Now, consider $k \in \N$ such that $\alpha^k v = v$, for any $n \leq k$. Equation (\ref{df0}) holds for every $k \in \N$. Using the BCH series, we obtain:
		\begin{equation*}
			\exp{(df_0^{ \tau n_{k_1} }v)}\exp{(df_0^{\tau n_{k_2} }v)} = \exp{(df_0^{ \tau n_{k_1} } v + df_0^{ \tau n_{k_2} } v + O_1)}, 
		\end{equation*}
		where $O_1$ is given by the Brackets between $df_0^{\tau n_{k_1} }v$ and $df_0^{\tau n_{k_2} }v$, which is in $[\mathfrak{n}, \mathfrak{n}^j_{i, \alpha}] \subset [\mathfrak{n}, \mathfrak{n}_i] =  \mathfrak{n}_{i+1}$ and we have $O_1 \in \mathfrak{n}_{i+1}$. By  Lemma \ref{lema212}, we obtain  
		\begin{equation*}
			\exp{(df_0^{\tau n_{k_1} } v + df_0^{\tau n_{k_2} } v + O_1)} = \exp{(df_0^{\tau n_{k_1} } v + df_0^{\tau n_{k_2} } v)}g_1 
		\end{equation*}
		for some $g_1 \in N_{i + 1} \subset \mathcal{R}$. By item $(4)$ in the Proposition \ref{reachablesetprop} we get
		\begin{equation*}
			\exp{(df_0^{ \tau n_{k_1} }v)}\exp{(df_0^{ \tau n_{k_2} }v)} = f_0^{ \tau n_{k_1}  }(\exp{v})f_0^{ \tau n_{k_2}  }(\exp{v}) \in \mathcal{R}_{2 \tau n_{k_2} - \tau n_{k_1}} \subset \mathcal{R}.
		\end{equation*}
		
		Once $g_1,g_1^{-1} \in N_{i+1} \subset \mathcal{R}$, we obtain
		\begin{equation*}
			\exp{(df_0^{\tau n_{k_1} }v + df_0^{\tau n_{k_2} }v)} = \exp{(df_0^{\tau n_{k_1} }v + df_0^{\tau n_{k_2} }v + O_1)} g_1^{-1} \in \mathcal{R}\cdot g_1^{-1} \subset \mathcal{R}. 
		\end{equation*}
		
		By Proposition \ref{reachablesetprop} item $2$, we can choose $n_{k_3} \in \{n_{s_j}\}_{j \in \N}$ such that
		\begin{equation*}
			\exp{(df_0^{\tau n_{k_1} }v + df_0^{\tau n_{k_2} }v)} \in \mathcal{R}_{\tau n_{k_3}}. 
		\end{equation*}
		
		Using BCH series we have
		\begin{equation}\label{3elements}
			\exp{(df_0^{\tau n_{k_1} }v + df_0^{\tau n_{k_2} }v)} \exp{(df_0^{\tau n_{k_3} } v)} = \exp{(\sum_{i=1}^3 df_0^{\tau n_{k_i} } v + O_2)}, O_2 \in \mathfrak{n}_{i+1}. 
		\end{equation}
		
		Again by the Lemma \ref{lema212} we get 
		\begin{equation*}
			\exp{(\sum_{i=1}^3 df_0^{\tau n_{k_i} } v + O_2)} = \exp{(\sum_{i=1}^3 df_0^{\tau n_{k_i} } v)}g_2, g_2 \in N_{i+1}.
		\end{equation*}
		
		Then
		\begin{equation*}
			\exp{(df_0^{\tau n_{k_1} }v + df_0^{\tau n_{k_2} }v)} \exp{(df_0^{\tau n_{k_3} } v)} \in \mathcal{R}_{  \tau (n_{k_3} + n_{k_2} - n_{k_1})} \subset \mathcal{R}.
		\end{equation*}
		and again by the $f_0$-invariance of $N_{i+1}$ we have 
		\begin{equation*}
			\exp{(\sum_{i=1}^3 df_0^{\tau n_{k_i} } v)} =\exp{(\sum_{i=1}^3 df_0^{\tau n_{k_i} } v+ O_2)}g_2^{-1} \in \mathcal{R}\cdot g_2^{-1} \subset \mathcal{R}. 
		\end{equation*}
		
		Repeating the idea $m-2$ times, we get 
		\begin{equation*}
			\exp{(\sum_{i=1}^m df_0^{\tau n_{k_i}  } v)} \in \mathcal{R}. 
		\end{equation*}
		
		Summing up to $m$ in the expression $\alpha^{k} v= df_0^{k} v + u_k $ we get
		\begin{equation}\label{X}
			X = \sum_{i=1}^{m} \alpha^{\tau n_{k_j}} v = \sum_{i=1}^{m} df_0^{\tau n_{k_i}  } v  + u' 
		\end{equation}
		from the fact that $\sum_{i=1}^{m} \alpha^{\tau n_{k_j}} = m$, with $u' \in \mathfrak{n}_{i, \alpha}^{j-1}$  defined by $u'=\sum_{j=1}^{m} u_{\tau n_{k_j}}$. Again, using the BCH formula, we have
		\begin{equation*}
			\exp(X + O) = \exp(\sum_{i=1}^{m} df_0^{\tau n_{k_i}  } v ) \exp(u')
		\end{equation*}
		where $O$ is the Lie brackets of $\sum_{i=1}^{m} df_0^{\tau n_{k_i}  } v $ and $u'$ which as before, we obtain $O \in \mathfrak{n}_{i+1}.$ In addition, as $u' \in \mathfrak{n}_{i,\alpha}^{j-1}$ and using the $df_0$ invariance and the induction hypothesis we have that  
		\begin{equation*}
			f_0^k(\exp{u'}) = \exp{(df_0^k(u'))} \in \mathcal{R}, \forall k \in \Z. 
		\end{equation*}
		
		The Lemma \ref{AginA} guarantees that 
		\begin{equation}\label{x+o}
			\exp{(X + O)}  = \exp{\left(\sum_{i=1}^{m} df_0^{\tau n_{k_i}  }v\right)} \exp{u'} \in \mathcal{R}\cdot \exp{u'} \subset \mathcal{R}
		\end{equation}
		
		Considering that 
		\begin{equation*}
			\exp{(X + O)} = \exp(X)g, g \in N_{i+1}, 
		\end{equation*}
		we have 
		\begin{equation*}
			\exp(X) = \exp{(X + O)}g^{-1} \in \mathcal{R}\cdot g^{-1} \subset \mathcal{R}. 
		\end{equation*}
		by the invariance of $N_{i+1}$. This implies that $N_i \subset \mathcal{R}$, as requested. 
		
		To prove that $N \subset \mathcal{R}$, as $N_{l + 1} = \{e\} \subset \mathcal{R}$, then $N_l \subset \mathcal{R}$ by the previous affirmation. Using this reasoning $l$ times, we get $N \subset \mathcal{R}$.
	\end{proof}
	
	\begin{remark}In the previous lemma, we assume without loss of generality that $\theta \in 2 \pi \mathbb{Q}$. If we suppose otherwise, one can consider a sequence $\{n_k\}_k \subset \N$ such that $\alpha^{n_k} = \cos{n_k \theta} + i \sin{n_k \theta} \longrightarrow 1$ such that $n_k \longrightarrow \infty$ \cite[Lemma 13]{CCS}. Summing up to $m$ in the expression $\alpha^k v = df_0^k v + u_k$, we  have 
		\begin{equation*}
			X + \varepsilon X = \left(\sum_{i=1}^{m} df_0^{\tau n_{k_i}  }v\right) + u'. 
		\end{equation*}
		with $\varepsilon \longrightarrow 0$ a complex number. Therefore, $X$ could be decomposed as 
		\begin{equation*}
			X = v + v', 
		\end{equation*}
		with $v' = u' - \varepsilon X \in \mathfrak{n}_{i,\alpha}^{j}$. Considering the BCH series of $v$ and $v'$ we get 
		\begin{equation*}
			\exp{(X + O)} = \exp{(u)}\exp{(v')}, 
		\end{equation*}
		where $O$ is the Lie brackets between $v$ and $v'$ which still belongs to $N_{i+1}$. This implies the equality  (\ref{x+o}) and the proof  follows in the same manner as above.
		
		Moreover, in the lemma above, the automorphism $df_0$ does not need to be in the form $e^{\mathcal{D}_1} \circ ... \circ e^{\mathcal{D}_k}$, given that such an expression is only used in the next proposition. 
	\end{remark}

	Finally, for the solvable case and using the fact of $df_0 = e^{\mathcal{D}_1} \circ ... \circ e^{\mathcal{D}_k}$, for some $k \geq 1$, we have the following proposition.
	
	\begin{proposition}\label{solvableteorema}Let $H \subset G^0$ be a solvable connected $f_0-$invariant Lie subgroup of $G^0$. Then $H \subset \mathcal{R}$. 
	\end{proposition}
	
	\begin{proof} Let $\mathfrak{h}$ be the Lie subalgebra of $H$ on $\mathfrak{g}^0$. As $H$ is connected, we have that $\mathfrak{h}$ is $df_0-$invariant. Taking $\mathfrak{n}$ as the nilradical of $\mathfrak{h}$, by \cite[Proposition 2.18]{sanmartin2} we get $\mathcal{D}_j(\mathfrak{h}) \subset \mathfrak{n}$. By the proposition (\ref{propnilpotent}) we get $H \subset \mathcal{R}$.
	\end{proof}
	
	Now, we are finally able to prove the following controllability condition.   
	
	\begin{theorem}\label{teoG^+0}Let $G$ be a connected solvable Lie group and consider the system (\ref{linearcontrolsystem}) defined on $G$. If $\mathcal{R}$ is open, then $G^{+,0} \subset \mathcal{R}$.   
	\end{theorem}
	
	\begin{proof}
		Space $\mathfrak{g}^+$ is an unstable subspace associated with the differential $df_0$. Note that $G^+ = \exp{\mathfrak{g}^+}$ is nilpotent and take $g \in  G^+$, then exists an $X \in \mathfrak{g}^+$ such that $g = \exp{X}$. As $0 \in \mathfrak{g}^+$ is stable in negative time, there is a $k \in \N$ such that $df_0^{-k} X$ is as close as necessary of $0$ for $k$ large enough. By continuity, 
		\begin{equation*}
			f_0^{-k}(\exp{X}) = \exp{ (df_0^{-k} X)} \in \mathcal{R}. 
		\end{equation*}
		
		Then $g=\exp{X}\in f_0^{-k}(\mathcal{R}) \subset \mathcal{R}$, that is $G^+ \subset \mathcal{R}$. Solvability of $G$ implies that $G^0$ is solvable. Therefore $G^0 \subset \mathcal{R}$ by Lemma \ref{G+0inreach}. Thus $G^{+,0} \subset \mathcal{R}$.
	\end{proof}
	
	The final result of this section provides sufficient conditions for the controllability of  system $(\ref{linearcontrolsystem}).$ recall that the controllability of (\ref{linearcontrolsystem}) can be seen by proving that $G = \mathcal{R} \cap\mathcal{C} =\mathcal{R} \cap \mathcal{R}^*$. The following result presents  a sufficient condition for the controllability of discrete-time linear systems on solvable Lie groups. 
	
	\begin{theorem}\label{theoremcontrol}The system (\ref{linearcontrolsystem}) is controllable if $\mathcal{R}$ and $\mathcal{C}$ are open sets and every eigenvalue of $df_0$ has norm equal $1$. 
	\end{theorem}
	
	\begin{proof}
		From Theorem \ref{teoG^+0}, it follows that $G^{0,+} \subset \mathcal{R}$ and $G^{0,-} \subset \mathcal{R}^*=\mathcal{C}$. As $\mathfrak{g} = \mathfrak{g}^0$ then $G = \mathcal{R} \cap \mathcal{R}^*$. Then the system (\ref{linearcontrolsystem}) is controllable.
	\end{proof}
	
	\subsection{Controllability and accessibility of affine Lie group}
	
	In this subsection, based on  the previous results, we investigate the controllability of linear systems on the two-dimensional affine Lie group. The real abelian solvable Lie groups of dimension $2$ are $\R^2$, $\mathbb{T} \times \R$ and $\mathbb{T}^2$ (see e.g. \cite{onish}). In the non-abelian case, the unique (up to an isomorphism) real solvable lie group is the open half plane $G = \R^{+} \ltimes \R$, endowed with the product 
	\begin{equation*}
		(x_1,y_1) \cdot (x_2,y_2) = (x_1x_2, y_2 + x_2 y_1). 
	\end{equation*}
	
	The Lie group $(G, \cdot)$ is called affine group and denoted by $\hbox{Aff}(2,\R)$. In particular, the automorphisms of $\hbox{Aff}(2,\R)$ are given by 
	\begin{equation*}
		\phi(x,y) = (x,a(x-1) + dy),
	\end{equation*}
	with  $ d \in \R^*$ and $a \in \R$. Then, the linear systems on $\hbox{Aff}(2,\R)$ can be defined by the functions 
	\begin{equation}\label{sist2dim}
		f((x,y), u):=f_u(x,y) = (h(u)x, a(x-1) + dy + g(u)x),
	\end{equation}
	where $h: \R^m \rightarrow \R^+$ and $g: \R^m \rightarrow \R$ are $\mathcal{C}^{\infty}$ functions satisfying $h(0)=1$ and $g(0)=0$. In fact, note that $f_0(x,y) = (x,a(x-1) + dy)$ is an automorphism of $\hbox{Aff}(2,\R)$. Note that 
	\begin{equation*}
		f_0^{-1}(x,y) = \left(x,\frac{-a}{d}(x-1) + \frac{y}{d}\right). 
	\end{equation*}
	and
	\begin{eqnarray*}
		f_u(x,y) &=& (h(u)x, a(x-1) + dy + g(u)x)\\
		&=&(h(u),g(u))(x,a(x-1) + dy)\\
		&=&f_u(1,0)f_0(x,y),
	\end{eqnarray*}
	for each $u\in U$ and $(x,y)\in\hbox{Aff}(2,\R)$. Hence $f$ defines a linear system on $\hbox{Aff}(2,\R)$ given by
	\begin{equation}\label{generalsystem}
		x_{k+1} = f(x_k, u_k),\ k \in \N,\ u\in U,
	\end{equation}
	where  $U$ is assumed to be a compact and convex neighborhood of $0\in\mathbb{R}^m$. Moreover, 
	\begin{equation*}
		f_0^k(x,y) = \left(x, a(x-1)(\sum_{j=0}^{k-1}d^{j}) + d^ky\right), \forall k \geq 1. 
	\end{equation*}
	
	From this  equality and Proposition (\ref{prop52}), one can define the solution $\varphi(k,(x,y),u)$ for every $(x,y) \in \hbox{Aff}(2,\R)$ and $u\in\mathcal{U}$. In the identity we have 
	\begin{equation*}
		\varphi(k,(1,0),u) = \left\{
		\begin{array}{ccc}
			(1,0),&if& k=0\\
			(h(u_0),g(u_0)), &if& k = 1\\
			\hat{f}_k(u), &if& k \geq 2
		\end{array},
		\right.
	\end{equation*}
	with 
	\begin{equation*}
		\hat{f}_k(u) = \left(\prod_{j=0}^{k-1} h(u_{k-1-j}), -a\left(\sum_{j=0}^{k-2} d^j\right) + d^{k-1} g(u_0) + \sum_{j=0}^{k-2} \left(d^{k-2-j}(a+g(u_{j+1}))\prod_{i=0}^j h(u_i)\right)\right).
	\end{equation*}
	
	Then by Proposition (\ref{prop52}), we get 
	\begin{eqnarray*}
		\varphi(k,(x,y),u) &=& \bigg(\prod_{j=0}^{k-1} h(u_{k-1-j})x, d^{k-1}(a + g(u_0))x + \sum_{j=0}^{k-2} d^{k-2-j}(a + g(u_{j+1}))\prod_{i=0}^j h(u_i) x\\
		& & + d^k y - a (\sum_{j=0}^{k-1}d^{k-1-j})\bigg). 
	\end{eqnarray*}
	for every $(x,y) \in\hbox{Aff}(2,\R)$ and $u=(u_i)\in\mathcal{U}$. 
	
	\begin{remark} As  previously discussed, it is challenging to work with the solution of (\ref{generalsystem}) defined above. However, considering the case $h(u)=1$, for every $u=(u_i)\in \mathcal{U}$, the solution is given by 
		\begin{equation*}
			\varphi(k,(x,y),u) = \left(x, \sum_{j=0}^{k-1} d^{k-1-j}((a + g(u_{j}))x -a) + d^k y \right).
		\end{equation*}
		This means that for each $(x,y) \in\hbox{Aff}(2,\R)$, the solutions are contained in the set $\{x\} \times \R$. Then $\hbox{int}\mathcal{R}(x,y) = \emptyset$, for every pair $(x,y) \in \hbox{Aff}(2,\R)$. Therefore, the results in the previous sections do not apply to this case.
	\end{remark}
	
	Consider the linear system (\ref{generalsystem}) and define the vector fields
	\begin{eqnarray*}
		X^+_u(x) &=& \frac{\partial}{\partial v}\bigg|_{v=0} f_u^{-1} \circ f_{u+v}(x)\\
		X^-_u(x) &=& \frac{\partial}{\partial v}\bigg|_{v=0} f_u \circ f_{u+v}^{-1}(x)\\
		\hbox{Ad}_{u_k\ldots u_1}X^+_{u_0}(x) &=& (df_{u_k} \circ \cdots\circ f_{u_1})_e^{-1}X^+_{u_0}(f_{u_k} \circ \cdots \circ f_{u_1}(x)).\\
		\hbox{Ad}^{-1\cdots-1}_{u_k\ldots u_1}X^-_{u_0}(x) &=& (df_{u_k}^{-1} \circ\cdots \circ f_{u_1}^{-1})_e^{-1}X^-_{u_0}(f_{u_k}^{-1} \circ\cdots\circ f_{u_1}^{-1}(x)). 
	\end{eqnarray*}
	Moreover, define the sets
	\begin{eqnarray*}
		\Gamma^+=\{\hbox{Ad}_{u_k\cdots u_1}X^+_{u_0}: k \in \N, u_k,\ldots,u_0 \in U\}\\
		\Gamma^-=\{\hbox{Ad}^{-1\cdots-1}_{u_k\cdots u_1}X^-_{u_0}: k \in \N, u_k,\ldots,u_0 \in U\}
	\end{eqnarray*}
	
	Considering $\Gamma^+(x)$ the vector space spanned by the vectors in $\Gamma^+$ on $x$, Jakubcyk and Sontag \cite{sontag} proved the following accessibility criteria.
	
	\begin{theorem}Consider a discrete-time control system of   form (\ref{sistgeneral}) on a $n$-dimensional smooth manifold $M$, where the control range $U \subset \R^m$ is a compact and convex neighborhood of $0$ in $\R^m$ and $f_u: M \rightarrow M$ is a diffeomorphism for any $u \in U$.
		\begin{itemize}
			\item[1-] The system is forward accessible if and only if 
			\begin{equation*}
				\dim\Gamma^+(x) = n, \forall x \in M. 
			\end{equation*}
			\item[2-] The system is backward accessible if and only if 
			\begin{equation*}
				\dim\Gamma^-(x) = n, \forall x \in M. 
			\end{equation*}
			\item[3-] The system is accessible if both conditions are satisfied.
		\end{itemize}
	\end{theorem}
	
	For linear systems on $\hbox{Aff}(2,\R)$, the previous result can be applied   to obtain a sufficient condition for the accessibility of (\ref{generalsystem}).
	
	\begin{proposition}\label{accesscriteria}The system (\ref{generalsystem}) is accessible if $-a h'(0) \neq g'(0)(d-1)$ and $h'(0) \neq 0$.
	\end{proposition}
	
	\begin{proof}
		Given any $u \in U$ we have
		\begin{equation*}
			df_u = 
			\begin{bmatrix}
				h(u) & 0\\
				a+g(u) & d
			\end{bmatrix}
			\mbox{ and }
			(df_u)^{-1} = 
			\begin{bmatrix}
				\dfrac{1}{h(u)} & 0\\
				-\dfrac{a+g(u)}{d h(u)} & \dfrac{1}{d}
			\end{bmatrix}.
		\end{equation*}
		
		The compositions $f_u^{-1} \circ f_{u + v}(x,y)$ and $f_u \circ f_{u+v}^{-1}(x,y)$ are given by 
		\begin{eqnarray*}
			f^{-1}_u \circ f_{u+v}(x,y) &=& \left(\frac{h(u+v)}{h(u)}x,\frac{x}{d}\left(g(u+v) + a-\frac{h(u+v)}{h(u)}(a+g(u))\right) + y\right),\\
			f_u \circ f_{u+v}^{-1}(x,y) &=& \left(\frac{h(u)}{h(u+v)}x, y + x\left(\frac{g(u)}{h(u+v)} - \frac{g(u+v)}{h(u+v)}\right)\right).  
		\end{eqnarray*}
		
		Then 
		\begin{eqnarray*}
			X^+_u(x,y) &=& \left(\frac{h'(u)}{h(u)}x, \frac{x}{d}\left(\frac{h'(u)}{h(u)}(-a-g(u)) + g'(u)\right)\right), \\
			X^-_{u}(x,y) &=& \left(-\frac{h'(u)}{h(u)}x, -\frac{g'(u)}{h(u)}x\right).
		\end{eqnarray*}
		
		Besides, we have
		\begin{eqnarray*}
			\hbox{Ad}_{u_1} X^+_{u_0}(x,y) &=&
			\begin{bmatrix}
				\dfrac{h'(u_0)}{h(u_0)}x\\
				-\dfrac{a+g(u_1)}{dh(u_0)}h'(u_0)x - \dfrac{(a + g(u_0))h'(u_0) h(u_1)}{d^2 h(u_0)}x + \dfrac{g'(u_0)h(u_1)}{d^2}x
			\end{bmatrix},\\
			\hbox{Ad}^{-1}_{u_1} X_{u_0}^-(x,y) &=&
			\begin{bmatrix}
				\dfrac{h'(u_0)}{h(u_0)}x\\
				-\dfrac{a+g(u_1)}{h(u_0)h(u_1)}h'(u_0)x - \dfrac{dg'(u_0)}{ h(u_0) h(u_1)}x
			\end{bmatrix},
		\end{eqnarray*}
		\begin{eqnarray*}
			\hbox{Ad}_{u_2 u_1} X^+_{u_0}(x,y) = 
			\begin{bmatrix}
				\dfrac{h'(u_0)}{h(u_0)}x\\
				T_{21}(x,u_0,u_1,u_2)
			\end{bmatrix}\mbox{ and }
			\hbox{Ad}^{-1 -1}_{u_2 u_1} X^+_{u_0}(x,y) = 
			\begin{bmatrix}
				\dfrac{h'(u_0)}{h(u_0)}x\\
				S_{21}(x,u_0,u_1,u_2)
			\end{bmatrix},
		\end{eqnarray*}
		with 
		\begin{eqnarray*}
			T_{21}(x,u_0,u_1,u_2) &=& -\dfrac{a+g(u_2)}{dh(u_0)}h'(u_0)x - \dfrac{(a + g(u_1))h'(u_0) h(u_2)}{d^2 h(u_0)}x\\
			& &-\dfrac{(a + g(u_0))h'(u_0) h(u_1)h(u_2)}{d^3 h(u_0)}x+ \dfrac{g'(u_0)h(u_2)h(u_1)}{d^3}x,\\ S_{21}(x,u_0,u_1,u_2) &=& -\frac{(a+g(u_1))h'(u_0)x}{h(u_0)h(u_1)} - \frac{d(a+g(u_2))h'(u_0)x}{h(u_0)h(u_1)h(u_2)} - \frac{d^2 g'(u_0)x}{h(u_0)h(u_1)h(u_2)}.
		\end{eqnarray*}
		
		By taking $u_0 = u_1 = u_2 = 0$ in the above vector fields above we get 
		\begin{eqnarray*}
			\hbox{Ad}_{u_2 u_1} X^+_{u_0}(x,y) &=& 
			\begin{bmatrix}
				h'(0)x\\
				-\dfrac{a h'(0)x}{d}\left(\dfrac{1}{d^2} + \dfrac{1}{d} + 1\right) + \dfrac{g'(0)}{d^3}x 
			\end{bmatrix},\\
			\hbox{Ad}_{u_1} X^+_{u_0}(x,y) &=&
			\begin{bmatrix}
				h'(0)x\\
				-\dfrac{a h'(0)x}{d}\left(\dfrac{1}{d} + 1\right) + \dfrac{g'(0)}{d^2}x
			\end{bmatrix},
		\end{eqnarray*}
		\begin{eqnarray*}
			\hbox{Ad}^{-1 -1}_{u_2 u_1} X^-_{u_0}(x,y) &=& 
			\begin{bmatrix}
				h'(0)x\\
				-a h'(0)x - da g'(0)x - d^2 g'(0)x
			\end{bmatrix},
		\end{eqnarray*}
		and 
		\begin{eqnarray*}
			\hbox{Ad}^{-1}_{u_1} X^-_{u_0}(x,y) &=&
			\begin{bmatrix}
				h'(0)x\\
				-a h'(0)x - d g'(0)x
			\end{bmatrix}.
		\end{eqnarray*}
		
		If we consider the sets 
		\begin{eqnarray*}
			\alpha= \{\hbox{Ad}^{-1 -1}_{u_2 u_1} X^-_{u_0}(x,y),\hbox{Ad}^{-1}_{u_1} X^-_{u_0}(x,y)\} \hbox{ and }
			\beta =  \{\hbox{Ad}_{u_2 u_1} X^+_{u_0}(x,y),\hbox{Ad}_{u_1} X^+_{u_0}(x,y)\}
		\end{eqnarray*}
		where $u_0 = u_1 = u_2 = 0$, one can prove that $\alpha$ and $\beta$ are linear independent sets if, and only if, $h'(0) \neq 0$ and $-ah'(0) \neq g'(0)(d-1)$ as required.
	\end{proof}
	
	We now prove that this condition guarantees that $\mathcal{R}$ and $\mathcal{C}$ are open sets. 
	
	\begin{proposition}\label{opennessofR}If $-a h'(0) \neq g'(0)(d-1)$ and $h'(0) \neq 0$, then it is reachable and controllable sets are open.
	\end{proposition}
	
	\begin{proof}
		Initially, using the notation $f_u \circ f_v := f_{u,v}$, note that
		\begin{equation*}
			\frac{\partial}{\partial (u,v)} f_{u,v}(e) = 
			\begin{bmatrix}
				h'(u)h(v) & h'(v)h(u)\\
				g'(u)h(v) & ah'(v) + dg'(v) + h'(v)g(u)
			\end{bmatrix}
		\end{equation*}
		and for $u = v = 0$ one get 
		\begin{equation}\label{matrixaccess}
			\frac{\partial}{\partial (u,v)} f_{u,v}(e) = 
			\begin{bmatrix}
				h'(0) & h'(0)\\
				g'(0) & ah'(0) + dg'(0)
			\end{bmatrix}
		\end{equation}
		since $g(0) = 0$ and $h(0) = 1$. Then, the matrix above has rank $2$ if, and only if, $-a h'(0) \neq g'(0)(d-1)$ and $h'(0) \neq 0$. The matrix in (\ref{matrixaccess}) guarantees that $e \in \hat{\mathcal{R}}$, which is open. Therefore, as $\hat{\mathcal{R}} \subset \mathcal{R}$  we have $e \in \hbox{int}\mathcal{R}$. From Proposition (\ref{reachablesetprop}) item $6$, $\mathcal{R}$ is open. Analogously, we show that the same condition is   valid for the reversed counterpart system
		\begin{equation*}
			\Sigma': g_{k+1} = f_{u_k}^{-1}(e)f_0^{-1}(g_k)
		\end{equation*}
		
		Therefore $\mathcal{C}=\mathcal{R}^*$ is also open.
	\end{proof}
	
	We conclude this section by presenting a sufficient condition for the controllability of (\ref{generalsystem}), for the case in which  $f_0$ is an inner automorphism. 
	
	\begin{theorem}Consider the discrete-time linear system (\ref{generalsystem}) where $f_0$ is an inner automorphism. If $d=1$ and $ah'(0) \neq 0$, then (\ref{generalsystem}) is controllable. 
	\end{theorem}
	
	\begin{proof}
		In fact, the conditions $h'(0) \neq 0$ and $-ah'(0) \neq g'(0)(d-1)$ imply by  Theorem (\ref{opennessofR}) that the sets $\mathcal{R}$ and $\mathcal{C}$ are open. Furthermore, since $d$ and $1$ are the only eigenvalues of $df_0$, according to Theorem (\ref{theoremcontrol}) the assumption $d=1$ implies that the system in (\ref{generalsystem}) is controllable. 
	\end{proof}
	
	\begin{remark}In the previous results, we are considering $g:U \longrightarrow \R$ as a non-zero $\mathcal{C}^{\infty}$ function. 
	\end{remark}
	
	\section{Controllability on nilpotent Lie groups}\label{section4}
	
	In this section we present and kind of converse of Theorem (\ref{theoremcontrol}). It is presumed that $G$ is a connected, simply connected, and nilpotent Lie group. And in this case  each connected subgroups are closed and simply connected (cf. \cite[Corollary 1.134]{Knapp}).
	
	To achieve our goal, we need the following lemma. 
	
	\begin{lemma}\label{relativelycompact}
		Suppose that $df_0$ has no eigenvalue with absolute value greater than $1$ and $\mathcal{R}$ is open. If $M:=G/G^0$ admits a $G$-invariant metric, then $\mathcal{R}_{G^-}=\mathcal{R}\cap G^-$ is a relatively compact set.
	\end{lemma}
	
	\begin{proof}
		Consider the induced linear system
		\begin{equation*}
			x_{k+1}=\bar{f}(x_{k},u_{k}),\ x_{k}\in M,\ u=(u_{i})_{i\in\mathbb{N}_{0}%
			}\in\mathcal{U}, \label{induced}%
		\end{equation*}
		on $M:=G/G^0$ with solutions denoted by $\bar{\varphi}$ that satisfies $\pi(\varphi(k,g,u))=\bar{\varphi}(k,\pi(g),u)$,
		for all $k\in\mathbb{N}_{0}$, $g\in G$ and $u\in\mathcal{U}$, where $\pi:G\rightarrow G/G^0$ denotes the standard projection (see end of Section \ref{section2}).
		
		If $\varrho$ is the distance induced by the $G$-invariant Riemannian metric on $M$, then given $x,y\in M$ and a smooth curve $\gamma:[0,1]\rightarrow M$ with $\gamma(0)=x$ and $\gamma(1)=y$ we have that $\bar{f}_0^k\circ\gamma$ is a smooth curve with $\bar{f}_0^k\circ\gamma(0)=\bar{f}_0^k(x)$ and $\bar{f}_0^k\circ\gamma(1)=\bar{f}_0^k(y)$. Hence
		
		$$\varrho(\bar{f}_0^k(x), \bar{f}_0^k(y))\leq\displaystyle\int_{0}^{1}|(d\bar{f}_0^k)_{\gamma(t)}\gamma^{\prime}(t)|dt.$$
		
		Note that $\|(d\bar{f}_0^k)_{gG^0}\|=\|(d\bar{f}_0^k)_{eG^0}\|$ holds for all $x\in M$ since the metric is $G$-invariant. From \cite[Remark 9 and Proposition 10]{CCS1} there exists a $c\geq1$ and $\sigma\in(0,1)$ such that
		\begin{eqnarray*}
			\varrho(\bar{f}_0^k(x), \bar{f}_0^k(y))\leq \displaystyle\int_{0}^{1}\|(d\bar{f}_0^k)_{o}\||\gamma^{\prime}(t)|dt\leq c^{-1}\sigma^k	\displaystyle\int_{0}^{1}|\gamma^{\prime}(t)|dt, \\
		\end{eqnarray*}
		that is, 
		\begin{eqnarray}\label{stable}
			\varrho(\bar{f}_0^k(x), \bar{f}_0^k(y))\leq c^{-1}\sigma^k\varrho(x,y), \ \forall \ k\in\mathbb{N}.
		\end{eqnarray}
		Moreover, the solutions $\bar{\varphi}$ satisfy
		\begin{eqnarray}\label{cocycle}
			\bar{\varphi}(k+l, eG^0, u)=\mathcal{L}_{\varphi(k,e,\Theta_k(u))}(\bar{f}_0^k(\bar{\varphi}(l, eG^0, u))).
		\end{eqnarray}
		
		Now we prove by induction on $k$ that 
		$$\varrho(\bar{\varphi}(k,eG^0,u),eG^0)\leq c^{-1}a\displaystyle\sum_{i=0}^{k-1}\sigma^i,$$
		where $a:=\displaystyle\max_{u\in U}\varrho(\bar{f}_u(eG^0), eG^0)$. Assume without loss of generality that $c^{-1}\geq 1$.
		For $k=1$, it follows from (\ref{stable}) that
		\begin{eqnarray*}
			\varrho(\bar{\varphi}(1,eG^0,u),eG^0)=\varrho(\bar{f}_{u}(eG^0),eG^0)\leq c^{-1}a.
		\end{eqnarray*}
		Suppose that the assertion holds for $n\geq 1$. Equality (\ref{cocycle}) yields
		$$\bar{\varphi}(k+1,eG^0,u)=\mathcal{L}_{\varphi(k,e,\Theta_k(u))}(\bar{f}_{0}^k(\bar{f}_{u}(eG^0))),$$
		consequently,
		\begin{eqnarray*}
			\varrho(\bar{\varphi}(k+1,eG^0,u),eG^0)&\leq&\varrho(\mathcal{L}_{\varphi(k,e,\Theta_k(u))}(\bar{f}_{0}^k(\bar{f}_{u}(eG^0))), \mathcal{L}_{\varphi(k,e,\Theta_k(u))}(eG^0))\\
			&+&\varrho(\mathcal{L}_{\varphi(k,e,\Theta_k(u))}(eG^0), eG^0)\\
			&=&\varrho(\bar{f}_{0}^k(\bar{f}_{u}(eG^0)), eG^0)+\varrho(\bar{\varphi}(k,eG^0,u),eG^0),
		\end{eqnarray*}
		because $\mathcal{L}_g$ is an isometry and $\mathcal{L}_g(eG^0)=\pi(g)$ for all $g\in G$. By (\ref{stable}) and the inductive hypothesis we have
		\begin{eqnarray*}
			\varrho(\bar{\varphi}(k+1,eG^0,u),eG^0)&\leq&\varrho(\bar{f}_{0}^k(\bar{f}_{u}(eG^0)), eG^0)+\varrho(\bar{\varphi}(k,eG^0,u),eG^0)\\
			&\leq&c^{-1}\sigma^k\varrho(\bar{f}_u(eG^0),eG^0)+c^{-1}a\displaystyle\sum_{i=0}^{k-1}\sigma^i\\
			&\leq&c^{-1}a\sigma^k+c^{-1}a\displaystyle\sum_{i=0}^{k-1}\sigma^i=c^{-1}a\displaystyle\sum_{i=0}^{k}\sigma^i.
		\end{eqnarray*}
		
		Since $\sigma\in(0,1)$, one has $$\pi(\mathcal{R})=\pi\left(\displaystyle\bigcup_{\tau\in \N}\mathcal{R}_\tau\right)=\displaystyle\bigcup_{\tau\in \N}\pi(\mathcal{R}_\tau)\subset\overline{B_\delta(eG^0)},$$
		where $\delta=c^{-1}a\displaystyle\sum_{i=0}^{\infty}\sigma^i<\infty$. Hence $\pi(\mathbb{R})$ is relatively compact in $M$. Note also that $\pi|_{G^-}:G^-\rightarrow M$ is a homeomorphism because $G^{-,0} = G^-\cap G^0=\{e\}$ (c.f. \cite[Lemma 3.6]{AyalaeRomaneAdriano}). Therefore, $\mathcal{R}_{G^-}$ is relatively compact in $G^-$. This result follows because $G^-$ is closed in $G$.
	\end{proof}
	
	The next proposition establishes a necessary and sufficient condition for controllability, using the reachable set from the identity $\mathcal{R}$. 
	
	\begin{proposition}\label{propreachablefrome}
		Assume that $\mathcal{R}$ is open. Then $\mathcal{R}=G$ if, and only if, $G=G^{+,0}$.
	\end{proposition}
	
	\begin{proof}
		If $G=G^{+,0}$, then Theorem \ref{teoG^+0} implies $\mathcal{R}=G$. Reciprocally, assume $\mathcal{R}=G$. The proof that $G=G^{+,0}$   proceed by induction on $\dim G$. If $\dim G=1$, then $G$ is abelian and we can conjugate the system to its linearization by an exponential map, consequently the result follows from \cite[Corollary 16 (i)]{CCS}. Suppose that the assertion holds for any nilpotent connected and simply connected Lie group with dimension smaller than $d$. Consider a Lie group $G$ satisfying the hypothesis with $\dim G=n$. Since the center of a nilpotent Lie group is non-trivial, the Lie group $H:=G/Z(G)$ is nilpotent, connected and $\dim H<d$. 
		
		The $f_0$-invariance of $Z(G)$ implies that the induced linear system on $H$ is well-defined. By hypothesis, $\mathcal{R}=G$ yields $\pi(\mathcal{R})=H$, where $\pi:G\rightarrow H$ is the standard projection. By induction we have $H=H^{+,0}$. Lemma \ref{invdyngroup} implies that $G^-\subset Z(G)$ and by Proposition \ref{decomposable} we find that $G^+$ is a normal subgroup of $G$. If $\mathcal{R}^{-,0}$ denotes the reachable set of  the induced linear system on $G^{-,0}=G/G^+$, then $\mathcal{R}^{-,0}=G^{-,0}$. Moreover, $G^{-,0}/G^0$ admits a $G$-invariant Riemannian metric because $G^-\subset Z(G)$.  Then, $\text{Ad}(G^0)|_{\mathfrak{g}^-}=\{\text{id}_{\mathfrak{g}^-}\}$ is compact in $\text{Gl}(\mathfrak{g}^-)$. From Lemma \ref{relativelycompact}, $\mathcal{R}^{-,0}\cap G^-=G^-$ is relatively compact in $G$. As $G^-$ is closed then it is compact. Hence by Proposition \ref{compactsubgroup} $G^-=\{e\}$, therefore $G=G^{+,0}$.
	\end{proof}
	
	\begin{remark} In the proof of Proposition (\ref{propreachablefrome}), the assumption that $G$ is nilpotent is employed to establish the center, Z(G), as a non-trivial set.
	\end{remark}
	
	Finally, the main result of this section can be stated.
	
	\begin{theorem}\label{controlnilpotent}
		Consider the linear system (\ref{linearcontrolsystem}) on a nilpotent connected and simply connected Lie group $G$. Then, the system is controllable if, and only if, $\mathcal{R}$ and $\mathcal{C}$ are open and $G=G^0$.
	\end{theorem}
	
	\begin{proof}
		If we assume that $\mathcal{R}$ and $\mathcal{C}$ are open and $G=G^0$, then the result follows from Theorem \ref{theoremcontrol}. However, if the system is controllable, then $\mathcal{R}$ and $\mathcal{C}$ are clearly open. Moreover, since $G=\mathcal{R}\cap\mathcal{C}$, we can apply Proposition \ref{propreachablefrome} to $\mathcal{R}$ and $\mathcal{R}^*=\mathcal{C}$  to conclude that $G=G^0$.
	\end{proof}
	
	\begin{remark}
		Theorem \ref{controlnilpotent} generalizes \cite[Corollary 16 (iii)]{CCS} regarding the controllability of linear systems on euclidean spaces.
	\end{remark}
	
	\begin{example}
		Consider the Heisenberg group 
		$$\mathbb{H}=\left\{\begin{bmatrix}
			1 & x_2 & x_1\\
			0 & 1 & x_3\\
			0 & 0 & 1
		\end{bmatrix};\ x_1,x_2, x_3\in\mathbb{R}\right\},$$
		which is diffeomorphic to $\mathbb{R}^{3}$
		with the product
		$$(x_{1},x_{2},x_{3})\cdot(y_{1},y_{2},y_{3})=(x_{1}+y_{1}+x_{2}y_{3},x_{2}+y_{2},x_{3}+y_{3}).$$
		Let $U$ be a compact and connected neighborhood of $0\in\mathbb{R}$ and $f:\mathbb{H}\times U\rightarrow\mathbb{H}$ given by
		$$f_{u}(x_{1},x_{2},x_{3})=\left(  x_{1}+x_{2}+\dfrac{x_{2}^{2}}{2}
		+ux_{2}+ux_{3}-\dfrac{u}{2}-\dfrac{u^{2}}{3},x_{2}+u,x_{2}+x_{3}-\dfrac{u}
		{2}\right).$$
		It is not difficult to see that 
		\begin{eqnarray}\label{heisenberg}
			g_{k+1}=f_{u_k}(g_{k}), \ u_k\in U
		\end{eqnarray}
		is a linear system on $\mathbb{H}$ (see \cite[Example 6]{CCS1}). Note that 
		$$f_u^{-1}(x_1,x_2,x_3)=\left(x_1-x_2-\frac{x_2^2}{2}+ux_2-ux_3+\frac{3u}{2}-\frac{2u^2}{2},x_2-u,-x_2+x_3-\frac{3u}{2}\right).$$
		Since each coordinate of both $f_u(0,0,0)$ and $f_u^{-1}(0,0,0)$ are polynomials at $u$ with root $u=0$ and non-null derivative in $u=0$, there is $\varepsilon>0$ such that $B_\varepsilon(0,0,0)\subset f_{\mathrm{int}U}(0,0,0)\cap f_{\mathrm{int}U}^{-1}(0,0,0)$ is a neighborhood of $(0,0,0)$ in $\mathcal{R}$ and $\mathcal{R}^*$, respectively. According to item (6) of Proposition \ref{reachablesetprop}, $\mathcal{R}$ and $\mathcal{R}^*=\mathcal{C}$ are open. Moreover, the unique eigenvalue of $df_0$ is $1$ (see \cite[Example 20]{CCS1}), hence $G=G^0$. Then Theorem \ref{controlnilpotent}, guarantees that the system (\ref{heisenberg}) is controllable.
	\end{example}
	
	\section{Declarations}

	\subsection{Conflict of interest/Competing interests}
	
	The authors declare that they have NO
	affiliations with or involvement in any organization or entity with any
	financial interest in the subject matter or materials discussed in this
	manuscript.
	
	\subsection{Ethics approval}
	
	Not applicable

	\subsection{Availability of data and materials}
	
	Not applicable
	
	\subsection{Code availability}
	
	Not applicable
	
	\subsection{Authors' contributions}
	
	All authors contributed to all sections. All authors reviewed the final manuscript.

\end{document}